\documentclass{amsart}

\usepackage{amsmath,amsthm,amsfonts,amscd,amssymb} 
\usepackage{verbatim}      	% Allows quoting source with commands.
\usepackage{epsfig}         	% Allows inclusion of eps files.
\usepackage{url}		% Allows good typesetting of web URLs.
\usepackage{epic,eepic,latexsym}
\usepackage{graphicx}
\usepackage{color}
\usepackage{lcd}
\usepackage{mathrsfs}
\usepackage[centering,text={15.5cm,22cm},
        marginparwidth=20mm,dvips]{geometry}
\usepackage[linktocpage]{hyperref}
\usepackage{lscape}
\usepackage{tabularx}
\usepackage{marginnote}
\usepackage[all]{xy}
\usepackage{enumitem}
\usepackage{stmaryrd}
\usepackage{upgreek}

\DeclareMathAlphabet{\mathpzc}{OT1}{pzc}{m}{it}

\DeclareMathOperator{\Spec}{Spec}

\DeclareMathOperator{\NE}{NE}

\DeclareMathOperator{\trop}{trop}

\DeclareMathOperator{\vir}{vir}
\DeclareMathOperator{\ev}{ev}

\DeclareMathOperator{\pt}{pt}

\DeclareMathOperator{\QH}{QH}

\DeclareMathOperator{\forget}{Forget}
\DeclareMathOperator{\naive}{naive}

\DeclareMathOperator{\Mor}{Mor}

\let\?=\overline
\let\:=\colon
\let\llb=\llbracket
\let\rrb=\rrbracket
\let\bb=\mathbb

\let\rar=\rightarrow

\let\s=\mathcal

\let\wh=\widehat
\let\wt=\widetilde

\newcommand {\qq} {{\bf q}}

\theoremstyle{plain}% default
 \newtheorem{thm}{Theorem}[section]
 
 \newtheorem{lem}[thm]{Lemma}

  \newtheorem{conj}[thm]{Conjecture}

\theoremstyle{definition}
 \newtheorem{dfn}[thm]{Definition}

 \newtheorem{eg}[thm]{Example}

\theoremstyle{remark} 
 \newtheorem{rmk}[thm]{Remark}

\newenvironment{myproof}[1][\proofname]{\proof[#1]\mbox{}}{\endproof}

\title{Fano mirror periods from the Frobenius structure conjecture}%\\ {\tiny \red{Preliminary version}}}
\author{Travis Mandel}
\address{School of Mathematics\\
University of Edinburgh\\
Edinburgh EH9 3FD\\
UK}
\email{Travis.Mandel{\char'100}ed.ac.uk}
\thanks{The author was supported by the National Science Foundation RTG Grant DMS-1246989, and later by the Starter Grant ``Categorified Donaldson-Thomas Theory'' no. 759967 of the European Research Council.}

\begin{document}

\begin{abstract}
The Fano classification program proposed by Coates-Corti-Galkin-Golyshev-Kasprzyk is based on the mirror symmetry prediction that the regularized quantum period of a Fano should be equivalent to the classical period of its mirror Landau-Ginzburg potential.  We prove that this mirror equivalence follows from versions of the Frobenius structure conjecture of Gross-Hacking-Keel.  We also find that the regularized quantum period, which is defined in terms of descendant Gromov-Witten numbers, is in fact given by certain naive curve counts.
\end{abstract}

\maketitle

\setcounter{tocdepth}{1}
%\tableofcontents  

\setcounter{tocdepth}{1}
%\tableofcontents  

\section{Introduction}\label{Intro}

One of the foundational ideas in the Fanosearch program \cite{CCGGK} is the notion that one might be able to classify Fano's by instead classifying their mirror Landau-Ginzburg potentials.  Associated to a Fano $Y$ is a ``regularized quantum period'' $\wh{G}_Y$ defined in terms of certain descendant Gromov-Witten numbers of $Y$, cf. Definition \ref{qDef}.  We prove  that these descendant Gromov-Witten numbers are in fact given by certain naive counts of rational curves in $Y$, cf. \textbf{Theorem \ref{ThmNaive}}.

On the other hand, associated to a Landau-Ginzburg potential $W$ is a ``classical period'' $\pi_W$ defined in terms of the constant terms of powers of $W$, cf. Definition \ref{cper1}.  One says that $Y$ and $W$ are mirror if $\wh{G}_Y=\pi_W$.  The Frobenius structure conjecture \cite[arXiv v1, Conj. 0.8]{GHK1} and the constructions of \cite{CPS} suggest a precise way to construct the mirror potential $W$.  Under an additional toric transversality hypothesis (shown by the author to hold for cluster varieties  \cite{ManFrob}), or using a naive-counting version of the Frobenius structure conjecture (which Keel-Yu will show holds for all affine log Calabi-Yau varieties containing a Zariski open algebraic torus \cite{KY}), we show that this candidate $W$ really is mirror to $Y$, cf. \textbf{Theorem \ref{FanoThm}}.  This may be viewed as an algebro-geometric analog of \cite[Thm. 1.1]{Tonk}.  One hopes that combining our results with those of \cite{GSInt2} will yield this mirror symmetry result for all Fano's.

\subsection{Regularized quantum periods}\label{Qper}

By a \textbf{Fano variety}, we shall mean a smooth compact complex variety $Y$ such that the anti-canonical bundle $-K_Y$ is ample.  More generally, a \textbf{Fano orbifold} is a smooth integral separated Deligne-Mumford stack $Y$ which is proper and finite type over $\Spec \bb{C}$ and for which $-K_Y$ is ample.  Let $A_*(Y)$ denote the integral Chow lattice of $Y$ (cf. \cite{EG,Kr}), and  let $\NE(Y)$ denote the cone in $A_1(Y)$ generated by classes of effective curves.

Now given a Fano orbifold $Y$ and $\beta \in \NE(Y)$, let $\s{M}_{0,1}(Y,\beta)$ be the moduli space of stable maps $\varphi:(C,x)\rar Y$, where $C$ is a curve of genus $0$ with a single marked point $x$ and with $\varphi_*[C]=\beta$.  We refer to $\beta$ as the \textbf{class} of the maps $\varphi$, and we refer to
\begin{align}\label{dbeta}
    d_{\beta}:=\beta.[-K_Y]
\end{align} as the \textbf{degree}.  The moduli space\footnote{The orbifold version of Fano mirror periods appeared in \cite[Conj. B]{ACC}.  Recall that Gromov-Witten theory has been developed in the orbifold setting in \cite{AGV}.  There, one allows stable maps to have some prescribed orbifold structure on the domain, and so \cite{ACC}'s quantum period has extra coefficients coming from the ``twisted sector'' of the cohomology of $Y$.  We will not consider contributions from these classes and prescribed orbifold structures, and so our regularized quantum period is obtained from that of \cite{ACC} by setting their variables $x_i$ equal to $0$.} has a virtual fundamental class $[\s{M}_{0,1}(Y,\beta)]^{\vir}$ of dimension $\dim Y + d_{\beta}-2$.  Let $\ev_x:\s{M}_{0,1}(Y,\beta)\rar Y$ be the evaluation map $[\varphi:(C,x)\rar Y]\mapsto \varphi(x)$.

 Let $\pi:\s{C}\rar \s{M}_{0,1}(Y,\beta)$ denote the universal curve over the moduli space.  Let $\omega_{\pi}$ denote the relative cotangent bundle of $\pi$, and let $\sigma_x$ denote the section of $\pi$ corresponding to $x$.  Define
\begin{align}\label{psic1}
    \s{L}_x:=\sigma_x^* \omega_{\pi} \quad \mbox{and} \quad \psi_{x}:=c_1(\s{L}_x).
\end{align}

\begin{dfn}\label{qDef}
The \textbf{regularized quantum period} of $Y$ is the power series $\wh{G}_Y:=\sum_{\beta \in \NE(Y)} p_\beta z^{\beta}\in \bb{Q}\llb\NE(Y)\rrb$, where $p_0:=1$, $p_{\beta}:=0$ when $d_{\beta}=1$, and for $d_{\beta}\geq 2$,
\begin{align}\label{pd}
p_{\beta}:=(d_{\beta})!\int_{[\s{M}_{0,1}(Y,\beta)]^{\vir}} \psi_x^{d_{\beta}-2} \ev_x^*([\pt]).
\end{align}
\end{dfn}

\begin{thm}\label{ThmNaive}
Let $D\in |-K_Y|$ be a reduced normal crossings divisor which contains no orbifold points of $Y$.  If $D$ is smooth, or if either Conjecture \ref{FrobConjNaive} or \ref{ttconj} holds for $(Y,D)$, then the coefficient $p_{\beta}$ as in \eqref{pd} is equal to the number of class $\beta$ maps $[\varphi:\bb{P}^1\rar Y]$ with $\varphi(0)$ equal to a generically specified point $y$, and with $\varphi^{-1}(D)$ a generically specified collection of $d_{\beta}$ points in $\bb{P}^1$.  In particular, $\wh{G}_Y$ has positive integer coefficients.
\end{thm} 

One may apply the map $z^{\beta}\mapsto t^{d_{\beta}}$ to obtain a version of $\wh{G}_Y$ in $\bb{Q}\llb t\rrb$ as in \cite{CCGGK}, but it will present no extra difficulties for us to work with this slightly richer version of periods.  This will also prove advantageous when viewing classical periods as Laurent polynomials, cf. Remark \ref{formalCoeff}.

\begin{rmk}
We note that $\wh{G}_Y(t)$ could alternatively be defined by setting $$p_{\beta}=(d_{\beta})!\int_{[\s{M}_{0,3}(Y,\beta)]^{\vir}} \psi_{x_1}^{d_{\beta}} \ev_{x_1}^*([\pt])$$ where $\s{M}_{0,3}(Y,\beta)$ is the moduli space of stable maps with $3$ marked points $x_1,x_2,x_3$.  The equivalence for $d_{\beta}\geq 1$ follows from the Fundamental Class Axiom, while the equivalence for $d_{\beta}=0$ follows from the Point Mapping Axiom (cf. \cite[pg. 39]{Gr} for a nice list of axioms for Gromov-Witten theory). 
\end{rmk}

\subsection{The Frobenius structure conjecture, naive counting version}
As in Theorem \ref{ThmNaive}, let $Y$ be a Fano orbifold and let $D\in |-K_Y|$ be a reduced normal crossings divisor containing no orbifold points of $Y$.  We call such $(Y,D)$ a \textbf{Fano pair}.  If $D$ contains a $0$-stratum, we say $(Y,D)$ has \textbf{maximal boundary}.  We will not assume that $(Y,D)$ has maximal boundary unless otherwise stated.

A simple toric blowup $\eta:(\wt{Y},\wt{D})\rar (Y,D)$ is a blowup $\eta:\wt{Y}\rar Y$ of $Y$ along a stratum of $D$, with $\wt{D}$ the reduced inverse image of $D$.  A \textbf{toric blowup} is then a sequence of simple toric blowups.  Let $H^0_{\log}(Y,D)$ denote the free Abelian group generated by $[Y]$ and $[kD']$ for $k\in \bb{Z}_{>0}$ and $D'$ an irreducible component of $\wt{D}$ for some toric blowup $(\wt{Y},\wt{D})$, up to equivalence.  Here, for $\eta_i:(Y_i,D_i)\rar (Y,D)$, $i=1,2$ two toric blowups of $(Y,D)$, we view an irreducible component $D'_1\subset D_1$ as equivalent to an irreducible component $D'_2\subset D_2$ if they correspond to the same valuation on the function field of $Y$.

We refer to these generators $[Y]$ and $[kD']$ of $H^0_{\log}(Y,D)$ as \textbf{prime fundamental classes}, and we let $B(\bb{Z})$ denote the set of prime fundamental classes of $(Y,D)$.  Equivalently, $B(\bb{Z})$ is the set of integral points of the tropicalization of $(Y,D)$.  Let $$\QH^0_{\log}(Y,D):= \bb{Q}[\NE(Y)]\otimes H^0_{\log}(Y,D).$$  For $q\in B(\bb{Z})$, we write $\vartheta_{q}$ for the corresponding element $1\otimes q\in \QH^0_{\log}(Y,D)$.  If $q$ is represented by $[kD']$ as above, we denote $|q|:=k$, and if $q=[Y]$, we define $|q|:=0$.

For $q_1,\ldots,q_s\in B(\bb{Z})$, let $\eta:(\wt{Y},\wt{D})\rar (Y,D)$ be a toric blowup in which each $q_i$ is either $[Y]$ or is represented by $[k_iD_i]$ for $D_i$ an irreducible component of $\wt{D}$.  Let $(C,x_1,\ldots,x_s,x_{s+1},x_{s+2})$ be a generically specified irreducible genus $0$ curve with $s+2$ marked points.  Let $y$ be a generically specified point of $Y\setminus D$.  Define $N_{\beta}^{\naive}(q_1,\ldots,q_s)\in \bb{Z}_{\geq 0}\cup \{\infty\}$ to be the number of isomorphism classes of maps $\varphi:C\rar Y$ such that $\varphi_*[C]=\beta$, $\varphi(x_{s+1})=y$, and $$\varphi^* \s{O}_{\wt{Y}}(\wt{D})=\s{O}_C\left(\sum_{i=1}^s |q_i|x_i\right).$$  This last condition means that, for $i=1,\ldots,s$, if $q_i=[k_iD_i]$, then $\varphi(C)$ intersects $D_i$ at $x_i$ with order $k_i$, and furthermore, these account for all intersections of $\varphi(C)$ with $\wt{D}$.  

\begin{conj}[Frobenius structure conjecture---weak naive version]\label{FrobConjNaive}
For $q_1,\ldots,q_s$ prime fundamental classes of $(Y,D)$, each $N_{\beta}^{\naive}(q_1,\ldots,q_s)$ is finite, as is the sum
 \begin{align}\label{spoint0}
     \langle \vartheta_{q_1},\ldots,\vartheta_{q_s}\rangle^{\naive}:= \sum_{\beta \in \NE(\wt{Y})} z^{\eta_* \beta} N_{\beta}^{\naive}(q_1,\ldots,q_s)\in \bb{Z}[ \NE(Y)].
 \end{align}
Extend $\langle \cdot \rangle^{\naive}$ to define a  $\bb{Q}[\NE(Y)]$-multilinear $s$-point function $$\langle \cdot \rangle^{\naive}:\QH_{\log}^0(Y,D)^{s} \rar \bb{Q}[\NE(Y)].$$  Then there exists a product $*$ on $\QH_{\log}^0(Y,D)$ making it into a commutative associative $\bb{Q}[\NE(Y)]$-algebra with identity $\vartheta_{[Y]}$ such that 
\begin{align}\label{Naive-def}
    \langle \vartheta_{q_1},\ldots,\vartheta_{q_s}\rangle^{\naive} = \langle \vartheta_{q_1}*\cdots *\vartheta_{q_s}
    \rangle^{\naive}
\end{align}
for all $s$-tuples $\vartheta_{q_1}, \ldots, \vartheta_{q_s}$, $s\geq 1$.  Furthermore, $\langle \vartheta_q\rangle^{\naive}$ equals $1$ if $q=[Y]$ and $0$ otherwise, so the right-hand side of \eqref{Naive-def} can be viewed as taking the $\vartheta_{[Y]}$-coefficient of the product.
\end{conj}

We now give the cases where Conjecture \ref{FrobConjNaive} is known to hold.  By a \textbf{cluster Fano pair}, we mean a cluster log pair $(Y,D)$ as in \cite[\S 2]{ManFrob} for which $Y$ is Fano.  This is roughly a Fano pair $(Y,D)$ with maximal boundary such that $Y\setminus D$ is (up to codimension $2$) a fiber or subfamily of a cluster $\s{X}$-variety as in \cite{FG1} or \cite{GHK3}.  These include cases where $Y$ is only smooth as an orbifold.

\begin{thm}[\cite{KY} and \cite{ManFrob}]\label{PastNaiveThm}
Conjecture \ref{FrobConjNaive} holds for all cluster Fano pairs \cite[Thm. 1.1]{ManFrob}, and also for all Fano pairs with maximal boundary $(Y,D)$ such that $Y\setminus D$ is a smooth variety containing a Zariski open algebraic torus \cite{KY}.
\end{thm}

We describe Conjecture \ref{FrobConjNaive} as a ``weak'' version of the Frobenius structure conjecture because we do not include the usual requirement that $*$ is uniquely determined by \eqref{Naive-def} (although this uniqueness does hold in the cases from \cite{KY} and \cite{ManFrob} as in Theorem \ref{PastNaiveThm}).  We note that the condition that $D$ is ample can also be significantly weakened, e.g., replaced with the condition that $D$ supports an ample divisor.  This version is ``naive'' because each $N_{\beta}^{\naive}(q_1,\ldots,q_s)$ is given by naive counts of rational curves, rather than by the descendant log Gromov-Witten counts which we explain next.

\subsection{The Frobenius structure conjecture, Gromov-Witten version}

By a \textbf{tropical degree}, we mean a map $\Delta:J\rar B(\bb{Z})$ for some finite index-set $J$. In particular, for $\qq$ an $s$-tuple of points $q_1,\ldots,q_s\subset B(\bb{Z})$, we consider \begin{align}\label{Deltapp}
     \Delta_{\qq}:\{1,\ldots,s,s+1,s+2\}\rar B(\bb{Z})
 \end{align} defined by $\Delta_{\qq}(i)=q_i$ for $i=1,\ldots,s$ and $\Delta_{\qq}(s+1)=\Delta_{\qq}(s+2)=[Y]$.  Given a tropical degree $\Delta$, let $\wt{Y}^{\dagger}=(\wt{Y},\wt{D})$ be a toric blowup with $\wt{Y}$ projective and such that each $\Delta(i)$ with $\Delta(i)\neq [Y]$ is represented by $[k_iD_i]$ for $k_i:=|\Delta(i)|$ and $D_i$ an irreducible component of $\wt{D}$.   For $\beta\in \NE(\wt{Y})$, let $\s{M}_{0,\Delta}^{\log}(\wt{Y}^{\dagger},\beta)$ denote \cite{GSlog,AC}'s moduli stack\footnote{Since any orbifold points of $\wt{Y}$ are away from the boundary, one can use \cite{AGV} to extend the construction of the relevant moduli stacks to our orbifold setting.  Cf. \cite[\S 5.5]{GPS} for similar considerations from the viewpoint of relative stable maps.} of basic/minimal stable log maps $\varphi^{\dagger}:C^{\dagger}\rar \wt{Y}^{\dagger}$ over $\Spec \bb{C}$ satisfying the following collection of conditions:
\begin{itemize}
\item $C$ has genus $0$,
\item $\varphi_*[C] = \beta$,
\item $C^{\dagger}$ has $\#J$ marked points $\{x_i\}_{i\in J}$,
\item For each $i\in J$ with $\Delta(i)\neq [Y]$,  $\varphi(x_i)\in D_{i}$.  Furthermore, if $t_1$ is the generator for the ghost sheaf of $\wt{Y}^{\dagger}$ at a generic point of $D_{i}$, and $t_2$ is the generator for the ghost sheaf of $C^{\dagger}$ at $x_i$, then $\varphi^{\flat}:t_1\mapsto k_it_2$.
\end{itemize} 
When the component of $C$ containing $x_i$ is not mapped entirely into $\wt{D}$, this last condition means that the intersection multiplicity of $\varphi(C)$ with $D_{i}$ at $x_i$ is equal to $k_i$. 

Let $\ev_{i}:\s{M}_{0,\Delta}^{\log}(\wt{Y}^{\dagger},\beta)\rar \wt{Y}$ be the evaluation map $[\varphi^{\dagger}:C^{\dagger}\rar \wt{Y}]\mapsto \varphi(x_i)$.   Let $\pi:\s{C}\rar \s{M}^{\log}_{0,\Delta}(\wt{Y}^{\dagger},\beta)$ denote the universal curve over the moduli space.  Let $\omega_{\pi}$ denote the relative cotangent bundle of $\pi$, and let $\sigma_i$ denote the section of $\pi$ corresponding to $x_i$.  Analogously to \eqref{psic1}, define
\begin{align}\label{psi-log}
    \s{L}_{i}^{\log}:=\sigma_i^* \omega_{\pi} \quad \mbox{and} \quad \wt{\psi}_{i}:=c_1(\s{L}_{i}^{\log}).
\end{align}
The stack $\s{M}_{0,\Delta_{\qq}}^{\log}(\wt{Y}^{\dagger},\beta)$ has a virtual fundamental class $[\s{M}_{0,\Delta}^{\log}(\wt{Y}^{\dagger},\beta)]^{\vir}$ of dimension $\dim(Y)+s-1$. We can now define the relevant log Gromov-Witten numbers:
\begin{align}\label{Nbeta}
    N_{\beta}(q_1,\ldots,q_s) := \int_{[\s{M}^{\log}_{0,\Delta_{\qq}}(\wt{Y}^{\dagger},\beta)]^{\vir}}   \ev_{s+1}^*[\pt]\cdot \wt{\psi}_{{s+1}}^{s-1}.
\end{align}
We next define a $\bb{Q}[\NE(Y)]$-multilinear $s$-point function $\langle \cdot \rangle$ on $\QH_{\log}^0(Y,D)$ via
 \begin{align}\label{spointfunction}
     \langle \vartheta_{q_1},\ldots,\vartheta_{q_s}\rangle:=\sum_{\beta\in \NE(\wt{Y})} z^{\eta_*(\beta)}N_{\beta}(\vartheta_{q_1},\ldots,\vartheta_{q_s}).
 \end{align}

\begin{conj}[The Frobenius structure conjecture---weak Gromov-Witten version]\label{FrobConj}
For any Fano pair $(Y,D)$, there exists a product $*$ on $\QH_{\log}^0(Y,D)$ making it into a commutative associative $\bb{Q}[\NE(Y)]$-algebra with identity $\vartheta_{[Y]}$ such that 
\begin{align}\label{spoint}
    \langle \vartheta_{q_1},\ldots,\vartheta_{q_s}\rangle = \langle \vartheta_{q_1}*\cdots *\vartheta_{q_s}\rangle
\end{align}
for all $s$-tuples $q_1, \ldots, q_s\in B(\bb{Z})$, $s\geq 1$.
\end{conj}

\begin{rmk}\label{logFCA}
Using the Fundamental Class Axiom, we can modify the above setup slightly.  For $s=1$, one has that $\langle \vartheta_q\rangle$ is $1$ if $q=[Y]$ and $0$ otherwise, so as in Conjecture \ref{FrobConjNaive}, the right-hand side of \eqref{spoint} can be viewed as taking the $\vartheta_{[Y]}$-coefficient of the product.  For $s\geq 2$, one can remove the marked point $x_{s+2}$, replacing the tropical degree $\Delta_{\qq}$ with $\Delta'_{\qq}:=\Delta_{\qq}|_{\{1,\ldots,s,s+1\}}$, and decreasing the power of $\wt{\psi}_{{s+1}}$ in \eqref{Nbeta} by $1$, yielding 
\begin{align*}
        N_{\beta}(q_1,\ldots,q_s) = \int_{[\s{M}^{\log}_{0,\Delta'_{\qq}}(\wt{Y}^{\dagger},\beta)]^{\vir}} \ev_{s+1}^*[\pt] \cdot \wt{\psi}_{{s+1}}^{s-2}.
\end{align*}
This is the version of $N_{\beta}$ used in \cite[arXiv v1, \S 0.4]{GHK1}, and it will be useful to us in \S \ref{PeriodSection}.  The advantage of the version of $N_{\beta}$ in \eqref{Nbeta} is just that it elegantly includes the $s=1$ cases.
\end{rmk}

We will want the curves from the Gromov-Witten counts of Conjecture \ref{FrobConj} to satisfy the following additional property:
\begin{conj}[Toric transversaility conjecture]\label{ttconj}
In the definition of $N_{\beta}(q_1,\ldots,q_s)$ in \eqref{Nbeta}, all curves of $\s{M}^{\log}_{0,\Delta_{\qq}}(\wt{Y}^{\dagger},\beta)$ which satisfy generic representatives of the conditions $\ev_{s+1}^*[\pt]$ and $\wt{\psi}_{s+1}^{s-1}$ are torically transverse.
\end{conj}
Here, when we say that $\varphi:=[\varphi^{\dagger}:C^{\dagger}\rar \wt{Y}^{\dagger}]$ satisfies a generic representative of $\ev_{s+1}^*[\pt]$, we just mean that $\varphi(x_{s+1})$ equals some genercally specified point $y$ of $\wt{Y}$.  We say $\varphi$ satisfies a generic representative of $\wt{\psi}_{s+1}^{s-1}$ if it is in the intersection of the supports of the divisors associated to $s-1$ generically specified rational sections of the line bundle $\s{L}^{\log}_{s+1}$, i.e., if it is in the support of a generically specified representative of the Euler class of $(\s{L}^{\log}_{s+1})^{\oplus (s-1)}$.

\begin{thm}[\cite{ManFrob}, Thm. 1.4 and Prop. 6.1]
Conjectures \ref{FrobConj} and \ref{ttconj} are satisfied for all cluster Fano pairs.
\end{thm}

\begin{rmk}
As announced in \cite[Thm. 2.2]{GSInt}, upcoming work of Gross and Siebert \cite{GSInt2} will construct an algebra of theta functions from pairs $(Y,D)$ as in Conjecture \ref{FrobConj} for all log Calabi-Yau varieties $(Y,D)$ (and for many other spaces as well).  They show that their theta functions satisfy \eqref{spoint} at least for $s\leq 3$, and one expects the higher-$s$ cases to hold as well.  We hope that Conjecture \ref{ttconj} or some similarly useful condition will also hold in this general setting.
\end{rmk}

\subsection{Classical Periods}\label{cper}
Suppose that either Conjecture \ref{FrobConjNaive} or Conjecture \ref{FrobConj} holds, and let $A$ denote the $\bb{Q}[\NE(Y)]$-algebra obtained by equipping $\QH_{\log}^0(Y,D)$ with the product $*$.  Given $f\in A$ and any $d\in \bb{Z}_{\geq 0}$, we can write $$f^d=\sum_{q\in B(\bb{Z})} c_{f,d,q} \vartheta_q$$ 
for uniquely determined coefficients $c_{f,d,q}\in \bb{Q}[\NE(Y)]$.  In particular, let $c_{f,d}:=c_{f,d,[Y]}$ denote the $\vartheta_{[Y]}$-coefficient of $f^d$.  We can write $c_{f,d}=\sum c_{\beta,f,d}z^{\beta}$, where the sum is over a finite collection of $\beta\in \NE(Y)$ and $c_{\beta,f,d}\in \bb{Q}$.  We assume $f$ is such that if $c_{\beta,f,d}\neq 0$, then $\beta.[D]\geq d$.

\begin{dfn}\label{cper1}
Consider $f\in A$ as above.  The \textbf{classical period} of $f\in A$ is defined to be
\begin{align*}
    \pi_f:=\sum_{d=0}^{\infty} c_{f,d} \in \bb{Q}\llb \NE(Y)\rrb.
\end{align*}
\end{dfn}
The assumption that $\beta.[D]\geq d$ whenever $c_{\beta,f,d}\neq 0$ ensures that $\pi_f$ is a well-defined formal series. The primary example we have in mind, essentially as proposed by \cite{CPS}, is $f=W:=\vartheta_{[D_1]}+\ldots+\vartheta_{[D_s]}$ where $D_1,\ldots,D_s$ are the irreducible components of $D$.  The assumption is indeed easily seen to hold for $f=W$.

We are now prepared to state our second main theorem:

\begin{thm}\label{FanoThm}
Suppose either Conjecture \ref{FrobConjNaive} holds or Conjectures \ref{FrobConj} and \ref{ttconj}
 hold for a Fano pair
$(Y,D)$.  Let $D_1,\ldots,D_s$ denote the irreducible components of $D$, and let $W:=\vartheta_{[D_1]}+\ldots+\vartheta_{[D_s]}$.  Then 
\begin{align}\label{GYpiW}
    \wh{G}_Y = \pi_W.
\end{align}
In particular, by \cite{ManFrob}, this holds for cluster Fano pairs, and by \cite{KY}, this holds for Fano pairs whose interior is a smooth variety containing a Zariski open torus.
\end{thm}

\subsubsection{Relation to the usual notion of classical periods}

Definition \ref{cper1} is somewhat different looking from the definition of a classical period in \cite{CCGGK}.  We explain the relationship now.  Let $f\in \bb{Q}[\NE(Y)][z_1^{\pm 1},\ldots,z_n^{\pm 1}]$ be a Laurent polynomial with coefficients in $\bb{Q}[\NE(Y)]$.  So each coefficient has the form $cz^{\beta}$ for some $c\in \bb{Q}$ and $\beta \in \NE(Y)$.  We assume that each $\beta$ here is nonzero.  Then as in \cite[Def. 3.1]{CCGGK} (but with more general coefficients), one defines the classical period of $f$ to be\footnote{Morally, the torus $\gamma:=\{|z_1|=\ldots=|z_n|=1\}$ should be the class of an SYZ fiber of the mirror $\Spec A$.  The form $\frac{dz_1}{z_1}\wedge \cdots \wedge \frac{dz_n}{z_n}$ can be viewed as the holomorphic volume form $\Omega$ on the mirror with log poles along the boundary, normalized so that $\int_{\gamma} \Omega=1$.}
\begin{align}\label{pifint}
\pi_f:=\left(\frac{1}{2\pi i}\right)^{n} \int_{|z_1|=\ldots=|z_n|=1} \frac{1}{1-f} \frac{dz_1}{z_1}\wedge \cdots \wedge \frac{dz_n}{z_n}.
\end{align}
Equivalently, by expanding $\frac{1}{1-f}$ as a power series in $\bb{Q}\llb \NE(Y)\rrb$ and repeatedly applying the residue theorem, we have $\pi_f= \sum_{d=0}^{\infty} c_{f,d}$, where $c_{f,d}\in \bb{Q}[\NE(Y)]$ now means the degree-zero term in the Laurent polynomial expansion of $f^d$.

Now, using a tropical description of the theta functions $\vartheta_q$ (e.g., in terms of broken lines), one obtains, for each generic point $Q$ in the tropicalization $Y^{\trop}$ of $(Y,D)$, a local chart $\varphi_Q:A\rar \bb{Q}[\Lambda]\llb \NE(Y)\rrb$ for $\Spec A$.  Here $\Lambda$ denotes the lattice of integral tangent vectors at $Q$, and choosing a basis identifies $\bb{Q}[\Lambda]$ with $\bb{Q}[z_1^{\pm 1},\ldots,z_n^{\pm 1}]$, where $n$ is the dimension of the tropicalization (so $n=\dim(Y)$ iff $(Y,D)$ has maximal boundary).  Note that the description of $\pi_f$ as in \eqref{pifint} in terms of constant coefficients of powers of $f$ still makes sense in this formal setting, although the terms $c_{f,d}$ might a priori live in the completion $\bb{Q}\llb \NE(Y)\rrb$.

It is always the case that $\varphi_Q(\vartheta_{[Y]})=1$, while for any $q\in B(\bb{Z})\setminus \{[Y]\}$, $\varphi_Q(\vartheta_q)$ will contain no constant terms.  It follows that taking the $\vartheta_{[Y]}$-coefficient of an element $f\in A$ is equivalent to taking the constant coefficient of $\varphi_Q(f)$ (in particular, each $c_{f,d}$ does in fact live in $\bb{Q}[\NE(Y)]$).  So the two notions of classical period agree!

\begin{rmk}\label{formalCoeff}
According to \cite[Conj. 1.10]{Tonk}, one expects such charts $A\rar \bb{Q}[\NE(Y)][\Lambda]$ as above but without the formal completion to correspond to monotone Lagrangian tori up to Hamiltonian isotopy.  So taking coefficients in $\bb{Q}\llb \NE(Y)\rrb$ rather than just in $\bb{Q}$ seems to allow us to define classical periods via \eqref{pifint} more generally than might otherwise be possible.
\end{rmk}

\begin{eg}
When $(Y,D)$ is a cluster Fano pair, the mirror $\Spec A$ is the Langlands dual cluster variety \cite[Thm. 1.4]{ManFrob}, meaning that $\Spec A$ contains many algebraic tori corresponding to the different clusters.  Restricting $W$ to any of these yields a finite Laurent polynomial, and the description of $\pi_W$ as in \eqref{pifint} applies.  Indeed, in these cases, a perturbation of $Y^{\trop}$ will contain a cell complex, essentially part of the cluster complex, and for generic $Q$ in the interior of one of these cells, $\varphi_Q$ corresponds to the inclusion of a corresponding cluster torus in $\Spec A$.  Choosing a different algebraic torus in $\Spec A$ will result in a different expression for $f$, e.g., related by some sequence of mutations.  One therefore considers Laurent polynomials up to mutation equivalence as in \cite{CCGGK}.
\end{eg}

\subsection{Further history and context}

Landau-Ginzburg mirrors for toric Fano manifolds first appeared in Hori-Vafa \cite{HV}.  A description of the superpotentials in terms of certain counts of Maslov index $2$ holomorphic disks was then proved in \cite{CO}.  The mirror correspondence $\wh{G}_Y=\pi_W$ as in \eqref{GYpiW} originally appeared in \cite{Goly}.  The tropical construction of $W$ is due to \cite{CPS}, and the application of this construction to understanding \eqref{GYpiW} was worked out in detail for $Y=\bb{P}^2$ in \cite{PrP2}.  A version of Theorem \ref{FanoThm} has been proven from the symplectic viewpoint in \cite[Thm. 1.1]{Tonk}, where $W$ is defined in terms of a choice of monotone Lagrangian torus in $Y$.  A much richer ``bulk-deformed'' version of superpotentials and such mirror correspondences was studied tropically for $\bb{P}^2$ in \cite{GrP2}, and the bulk-deformed version of \eqref{GYpiW} was recently proven for toric Fano surfaces in \cite{HLZ} using tropical disk counts and an open Gromov-Witten correspondence theorem.

\subsection*{Acknowledgements}

I wish to thank Alessio Corti, as his talks on \cite{CCGGK} at UT Austin in January, 2013 are what inspired this work.  I am also very grateful to Sean Keel and Mark Gross for keeping the author informed of their related works \cite{KY} and \cite{GSInt2}, respectively, and for additional helpful conversations.

\section{Proofs}\label{PeriodSection}

 We will repeatedly use the following result of Lehmann and Tanimoto:
\begin{lem}[\cite{LT}, Thm. 1.1]\label{LTlem}
Let $Y$ be a smooth projective weak Fano variety.\footnote{If $Y$ is only smooth as an orbifold, we can still apply Lemma \ref{LTlem} to a projective resolution of the associated singular scheme, as this is still a weak Fano variety.}   Then there is a proper closed subset $V\subsetneq Y$ such that any component of $\Mor(\bb{P}^1,Y)$ parametrizing curves not contained in $V$ will have the expected dimension.
\end{lem}

Thus, if $\sigma$ is an open subset of $\s{M}_{0,n}(Y,\beta)$ corresponding to maps $\varphi:(C,x_1,\ldots,x_n)\rar Y$ with $C$ irreducible and $\varphi(C)$ containing a general point of $Y$, then $\dim(\sigma)=n+\dim(Y)+d_{\beta}-3$.  We now show that curves contributing to $p_{\beta}$ as in \eqref{pd} are irreducible.

\begin{lem}\label{irr}
For $Y$ Fano, suppose $[\varphi:(C,x)\rar Y]\in \s{M}_{0,1}(Y,\beta)$ lies in $\ev_x^{-1}(y)$ for a generically specified point $y\in Y$.  Let $d:=d_{\beta}$ as in \eqref{dbeta}, and suppose furthermore that $[\varphi:(C,x)\rar Y]$ lies in the support of a generically specified representative of  $\psi_x^{d-2}=c_{d-2}(\s{L}_x^{\oplus(d-2)})$ for $\s{L}_x$ as in \eqref{psic1}.  Then $C$ is irreducible. 
\end{lem}

\begin{proof}
First note that because we have only one marked point, we cannot have a contracted component without having at least two non-contracted components as well.  So it suffices to prove that curves $\varphi:C\rar Y$ satisfying the generically specified conditions cannot have multiple non-contracted components.

By the Fano assumption, every non-contracted component contributes positively to the intersection with the boundary.  Hence, if $\varphi:C\rar Y$ does have multiple non-contracted components, then the component $C_x$ containing the marked point $x$ has degree strictly less than $d$.  If $\deg(C_x)=1$, then $C_x$ cannot contain a general point $y$ (using Lemma \ref{LTlem}), thus ruling out this case.  Similarly, if $d=2$ and $\deg(C_x)=0$, then $C$ must include two degree $1$ components which hit the general point $y$, again a contradiction.

Next consider the cases with $\deg(C_x)=0$ and $d>2$.  Let $C^0$ be the union of all components of $C$ contracted by $\varphi$, and let $C_x^0$ be the connected component of $C^0$ containing $C_x$.  Let $C_x'$ be the union of all components of $C$ which intersect $C_x^0$.  Let $C_1,\ldots,C_s$ be the components of $C_x'$ of positive degree, and let $d_i$ denote the degree of $C_i$.  Since $\varphi(C_i)$ intersects the generically specified point $y$, we must have each $d_i\geq 2$.

Let $\sigma$ be the smallest closed stratum of $\s{M}_{0,1}(Y,d)$ which contains this $[\varphi:(C,x)\rar Y]$, which we will denote as just $C$ for convenience.  Let $\rho:=\ev_x^{-1}(y)\cap\sigma$.  We have a map $\forget_{\rho}:\rho\rar \?{\s{M}}_{0,s+1}$ coming from contracting the components $C_1,\ldots,C_s$ of $C_x'$ and labelling the image of $C_i$ in $C_x^0$ by $x_i$.  
Since each $d_i$ is at least $2$, and since $\sum_{i=1}^s d_i\leq d$, we have $s\leq \lfloor\frac{d}{2}\rfloor$, and so $\dim \?{\s{M}}_{0,s+1}\leq \lfloor\frac{d}{2}\rfloor-2< d-2$.  But $\forget_{\rho}$ does not destabilize the curve-components containing $x$, so $\s{L}_x|_{\rho} = \forget_{\rho}^* \?{\s{L}}_x$ where $\?{\s{L}}_x$ denotes the corresponding $\psi$-class line bundle on $\?{\s{M}}_{0,s+1}$.  But the dimension implies that $\?{\psi}_x^{d-2}=0$ on $\?{\s{M}}_{0,s+1}$, where $\?{\psi}_x:=c_1(\?{\s{L}}_x)$.  Thus, $\rho$ cannot be in the support of a general representative of $\psi_x^{d-2}$, a contradiction.

Finally, suppose that $\deg(C_x)>1$.  As before, let $\sigma$ be the minimal closed stratum of $\s{M}_{0,1}(Y,d)$ which contains this $C$, and let $\rho=\ev_x^{-1}(y)\cap \sigma$.  We have a map $\forget_{\rho}:\rho\rar \s{M}_{0,1}(Y,\deg(C_x))$ which again does not destabilize the curve-components containing $x$. 
So we have as before that $\s{L}_x|_{\rho}=\forget_{\rho}^* \?{\s{L}}_x$, where $\?{\s{L}}_x$ now denotes the $\psi$-class line bundle on $\s{M}_{0,1}(Y,\deg(C_x))$.  By Lemma \ref{LTlem} and the point condition, the image of $\forget_{\rho}$ is contained in an open stratum of dimension strictly less than $d-2$, thus cannot satisfy a generic $\?{\psi}_x^{d-2}$-condition.  This gives our contradiction and completes the proof.
\end{proof}

\begin{lem}\label{tt}
Suppose that either Conjecture \ref{FrobConjNaive} or Conjecture \ref{ttconj} holds for $(Y,D)$.  Then all curves in $\s{M}_{0,1}(Y,\beta)$ which satisfy generically specified representatives of the point and $\psi$-class conditions are torically transverse.
\end{lem}
\begin{proof}
First suppose that Conjecture \ref{FrobConjNaive} holds.  Suppose $\varphi:(C,x)\rar Y$ satisfies the generically specified conditions but is not torically transverse.  By Lemma \ref{irr}, $C$ is irreducible, and so general points of $C$ map to $Y\setminus D$.  So there is some toric blowup $(\wt{Y},\wt{D})$ of $(Y,D)$ such that the proper transform of $\varphi(C)$ \textit{is} torically transverse.  However, the number $d'$ of points on $C$ which map to $\wt{D}$ is now strictly less than $d:=d_{\beta}$.  Now  Conjecture \ref{FrobConjNaive} implies that the point condition and the intersection multiplicities with the boundary determine (the proper transform of) $\varphi:(C,x)\rar Y$ up to at most $\dim(\s{M}_{0,d'+1})=d'-2<d-2$ many degrees of freedom, so there can be no such curves which also satisfy the generically specified $\psi_x^{d-2}$-condition, as desired.

Now suppose instead that Conjecture \ref{ttconj} holds for $(Y,D)$.  Given $C\in \s{M}_{0,1}(Y,\beta)$, we can lift $C$ to a log stable map in $\s{M}^{\log}_{0,\Delta'_{\qq}}(Y^{\dagger},\beta)$ for some $d$-tuple $\qq$.  Let $\forget: \s{M}^{\log}_{0,\Delta'_{\qq}}(Y^{\dagger},\beta)\rar \s{M}_{0,1}(Y,\beta)$  be the map forgetting the log structure and taking $x_{d+1}$ to $x$.  We have that $\wt{\psi}_{d+1}-\forget^* \psi_x$ is supported on the locus where the curve-component containing $x$ is destabilized by forgetting the marked points $x_1,\ldots,x_{d}$.  This is clearly disjoint from the locus $\ev_{d+1}^{-1}(y)$ for any $y\in Y\setminus D$.  Hence, if $C\in \s{M}_{0,1}(Y,D)$ satisfies generic representatives of the point and $\psi$-class conditions, then so do any log curves in $\forget^{-1}(C)$.  Such log curves are torically transverse by the assumption that Conjecture \ref{ttconj} holds, so $C$ must have been torically transverse as well, proving the claim.
\end{proof}

We are now ready to prove our main theorems.

\begin{myproof}[Proof of Theorems \ref{ThmNaive} and \ref{FanoThm}]
First note that the degree $d=0$ and $d=1$ cases of both theorems follow easily using the Fundamental Class Axiom as in Remark \ref{logFCA} and our previous observation that Lemma \ref{LTlem} implies there are no degree $1$ curves hitting a generic point $y\in Y$.  So from now on, we restrict to $d\geq 2$, and we use the Fundamental Class Axiom as in Remark \ref{logFCA} to forget the extra marked point $x_{d+2}$ in the definition \eqref{Nbeta} of $N_{\beta}(q_1,\ldots,q_d)$ for Conjecture \ref{FrobConj}.

By an $s$-partition $P$ of $d$, we mean a map $P:\{1,\ldots,s\}\rar \bb{Z}_{\geq 0}$ such that $P(1)+\ldots+P(s)=d$.  We write ${d\choose P}$ for the corresponding multinomial coefficient $\frac{d!}{P(1)!\cdots P(s)!}$.  So for $W=\vartheta_{[D_1]}+\ldots+\vartheta_{[D_s]}$ as in Theorem \ref{FanoThm}, we have
\begin{align*}
    W^d=\sum_{P|d} {d\choose P} \vartheta_{[D_1]}^{P(1)}\cdots \vartheta_{[D_s]}^{P(s)},
\end{align*}
where the sum is over all $s$-partitions $P$ of $d$.  Suppose Conjecture \ref{FrobConj} holds, so the $\vartheta_0$-coefficient $c_{W,d}$ of $W^d$ is given by
\begin{align}\label{sumPNbeta}
c_{W,d}=\sum_{P|d} {d\choose P} \sum_{\beta\in \NE(Y)}  z^{\beta}N_{\beta}(\qq_P),
\end{align}
where $\qq_P$ is a $d$-tuple consisting of $P(i)$ instances of $[D_i]$ for each $i=1,\ldots,s$.  Note that $N_{\beta}(\qq_P)=0$ unless $d_{\beta}=d$.  We want to show that $c_{W,d}$ equals 
$\sum_{\beta\in \NE(Y)} z^{\beta}  d!\int_{[\s{M}_{0,1}(Y,\beta)]^{\vir}} \psi_x^{d-2} \ev_x^*([\pt])$.

Forgetting the log structure of $(Y,D)$ induces a proper map $\forget:\s{M}^{\log}_{0,\Delta'_{\qq_P}}(Y^{\dagger},\beta)\rar \s{M}_{0,1}(Y,\beta)$ (taking the marking $x_{d+1}$ to the marking $x$).  Lemma \ref{tt} ensures that any curves contributing to $p_{\beta}$ are torically transverse, and by the assumption of Conjecture \ref{ttconj}, curves contributing to \eqref{sumPNbeta} are torically transverse as well.  The restriction of $\forget$ to the locus of torically transverse curves in $\s{M}^{\log}_{0,\Delta'_{\qq_P}}(Y^{\dagger},\beta)$ is generically finite over its image with degree $P(1)!\cdots P(s)!$, thus allowing us to apply the projection formula.  On the other hand, every torically transverse curve in $\s{M}_{0,1}(Y,\beta)$ is in $\forget(\s{M}^{\log}_{0,\Delta'_{\qq_P}}(Y^{\dagger},d))$ for precisely one $s$-partition $P$ of $d$, specifically, the partition with \begin{align}\label{Pbeta}
    P(i)=\beta.[D_{p_i}]
 \end{align} for each $i$.  When \eqref{Pbeta} is satisfied for each $i$, we write $\beta \in [P]$.
 
Since $\ev_x$ factors through $\forget$, it is clear that the point condition pulls back to the point condition.  Also, since $\forget$ does not destabilize torically transverse (or irreducible) curves, $\wt{\psi}_{{d+1}}$ and $\forget^* \psi_x$ agree on the relevant loci.  Furthermore, the local deformation/obstruction theory for torically transverse curves is the same whether we are in the log setting or not.  We can thus apply the projection formula to compute
\begin{align}\label{computation}
    \sum_{P|d} {d\choose P} \sum_{\beta \in \NE(Y)} z^{\beta} N_{\beta}(\qq_P) &= \sum_{P|d}\sum_{\beta \in \NE(Y)}z^{\beta}\frac{d!}{P(1)!\cdots P(s)!}\int_{[\s{M}^{\log}_{0,\Delta'_{\qq_P}}(Y^{\dagger},\beta)]^{\vir}} \wt{\psi}_{{d+1}}^{d-2}\ev_{d+1}^*[\pt] \nonumber \\
    &= \sum_{P|d}\sum_{\beta \in [P]}z^{\beta}\frac{d!}{P(1)!\cdots P(s)!}\int_{P(1)!\cdots P(s)![\s{M}_{0,1}(Y,\beta)]^{\vir}} \psi_x^{d-2}\ev^*[\pt]\\
    &= \sum_{\beta\in \NE(Y)}z^{\beta} d!\int_{[\s{M}_{0,1}(Y,\beta)]^{\vir}}\psi_x^{d-2}\ev^*[\pt],\nonumber 
\end{align}
as desired.  This proves Theorem \ref{FanoThm} in the cases where Conjectures \ref{FrobConj} and \ref{ttconj} hold.

In general, by Lemmas \ref{irr} and \ref{LTlem}, the curves in $\s{M}_{0,1}(Y,\beta)$ satisfying generic representatives of $\ev^*[\pt]$ and $\psi_x^{d-2}$ (for $d:=d_{\beta}$) are unobstructed, and so we may replace $[\s{M}_{0,1}(Y,\beta)]^{\vir}$  with the usual fundamental class $[\s{M}_{0,1}(Y,\beta)]$.  Now suppose that either $D$ is irreducible (so toric transversality holds automatically) or that either Conjecture \ref{FrobConjNaive} or Conjecture \ref{ttconj} holds (yielding toric transversality via Lemma \ref{tt}).  Then by the arguments of the previous paragraph, for $P$ the $s$-partition such that $[P]\ni \beta$, we obtain that $p_{\beta}$ equals ${d\choose P}$ times the number of torically transverse irreducible curves in $\s{M}^{\log}_{0,\Delta'_{\qq_P}}(Y^{\dagger},\beta)$ which are contained in the intersection of $\ev_{d+1}^{-1}(y)$ and the support of a general representative of $\wt{\psi}^{d-2}_{d+1}$.

To interpret the condition $\wt{\psi}^{d-2}_{d+1}$, let $\?{\forget}:\s{M}^{\log}_{0,\Delta'_{\qq_P}}(Y^{\dagger},\beta)\rar \?{\s{M}}_{0,s+1}$  denote the forgetful map which remembers only the stabilization of the domain of the log stable maps.  As before, $\s{L}_{d+1}^{\log}$ and $\forget^*\?{\s{L}}_{d+1}$ agree except possibly on the locus where the component of $C$ containing $x$ is contracted by the stabilization map, and since $d\geq 2$ this does not intersect the locus of irreducible torically transverse curves.  So we can replace $\psi_{d+1}^{d-2}$ with $\?{\forget}^* \?{\psi}_{d+1}^{d-2}$.  It is standard that $\?{\psi}_{d+1}^{d-2}$ in $\?{\s{M}}_{0,d+1}$ is the class of a point, so we may view $\?{\forget}^* \?{\psi}_{d+1}^{d-2}$ as generically specifying the image of the curve under $\?{\forget}$, i.e., specifying the domain marked curve.  Thus,
\begin{align}\label{pbNb}
p_{\beta}={d\choose P} N_{\beta}^{\naive}(\qq_P).
\end{align} 

Now, if Conjecture  \ref{FrobConjNaive} holds, it follows in the same way as \eqref{sumPNbeta} that
\begin{align}\label{sumPNbetaNaive}
c_{W,d}=\sum_{P|d} {d\choose P} \sum_{\beta\in \NE(Y)}  z^{\beta}N^{\naive}_{\beta}(\qq_P).
\end{align}
The Conjecture \ref{FrobConjNaive} case of Theorem \ref{FanoThm} now follows immediately from \eqref{sumPNbetaNaive} and \eqref{pbNb}.

Finally, note that in Theorem \ref{ThmNaive} we generically specify $\varphi^{-1}(D)$, and the factor ${d\choose P}$ in \eqref{pbNb} accounts for the number of ways to choose which points in $\varphi^{-1}(D)$ map to each component of $D$.  Theorem \ref{ThmNaive} thus follows from \eqref{pbNb}.
\end{myproof}

\bibliographystyle{amsalpha}  % Here the bibliography 		     %
\bibliography{main}        % is inserted.			     %
\index{Bibliography@\emph{Bibliography}}%

\end{document}